\documentclass[11pt,letterpaper,reqno]{amsart}

\usepackage{amsmath,amssymb,amsthm,enumitem}

\usepackage{multirow}
\usepackage{tikz}
\usepackage{amsfonts}

\usepackage{tikz}
\usetikzlibrary{matrix,arrows,decorations.pathmorphing}

\DeclareMathOperator{\Gal}{Gal}
\DeclareMathOperator{\Res}{Res}
\DeclareMathOperator{\Max}{Max}

\def\Tr{\operatorname{Tr}}
\def\N{\operatorname{N}}
\def\z{\zeta}
\def\a{\alpha}
\def\w{\omega}

\renewcommand{\phi}{\varphi}

\newtheorem{theorem}{Theorem}[section]
\newtheorem*{thm}{Theorem}
\newtheorem*{dfn}{Definition}

\newtheorem{proposition}[theorem]{Proposition}
\newtheorem{lemma}[theorem]{Lemma}
\newtheorem{corollary}[theorem]{Corollary}
\newtheorem{conjecture}[theorem]{Conjecture}
\theoremstyle{definition}
\newtheorem{remark}[theorem]{Remark}
\newtheorem{definition}[theorem]{Definition}
\newtheorem{example}[theorem]{Example}

\def\cqfd{
{\hfill
\kern 6pt\penalty 500
\raise -1pt\hbox{\vrule\vbox to 5pt{\hrule width 4pt
\vfill\hrule}\vrule}}
\break}

\font\tengoth=eufm10
\font\sevengoth=eufm7
\font\fivegoth=eufm5
\newfam\gothfam
\textfont\gothfam=\tengoth\scriptfont\gothfam=\sevengoth\scriptscriptfont\gothfam=\fivegoth
\def\goth{\fam\gothfam\tengoth}

\def\Sgoth{{\goth S }}

\title{Maximal differential uniformity polynomials 
}

\author[Aubry]{Yves Aubry}
\address[Aubry]{Institut de Math\'ematiques de Toulon - IMATH, Universit\'e de Toulon, France}
\address[Aubry]{Institut de Math\'ematiques de Marseille - I2M, Aix Marseille Univ, CNRS, Centrale Marseille, France}
\email{yves.aubry@univ-tln.fr}

\author[Herbaut]{Fabien Herbaut}
\address[Herbaut]{ESPE Nice-Toulon, Universit\'e de Nice Sophia-Antipolis, France}
\address[Herbaut]{Institut de Math\'ematiques de Toulon - IMATH, Universit\'e de Toulon, France}

\email{fabien.herbaut@unice.fr}

\author[Voloch]{Jos\'e Felipe Voloch}
\address[Voloch]{School of Mathematics and Statistics, University of Canterbury, Private Bag 4800, Christchurch 8140, New Zealand}
\email{felipe.voloch@canterbury.ac.nz}
\urladdr{http://www.math.canterbury.ac.nz/\~{}f.voloch}

\begin{document} 
\begin{abstract}
We provide  explicit infinite families of integers $m$
such that all the polynomials  of
${\mathbb F}_{2^n}[x]$ of degree $m$ have
maximal differential uniformity for $n$ large enough.
We also prove a conjecture of the third author for these families.
\end{abstract}

\date{\today}

\maketitle

\section{Introduction}

Throughout this paper
$n$ is a positive integer and $q=2^n$.
For a polynomial $f \in {\mathbb F}_{q}[x]$ 
we define the differential uniformity $\delta(f)$ following Nyberg (\cite{Nyberg}):
$$\delta(f):=\max_{(\alpha,\beta)\in{\mathbb F}_q^{\ast}\times{\mathbb F}_q}\sharp\{x\in{\mathbb F}_{q} \mid f(x+\alpha)+f(x)=\beta\}.$$

When $\delta(f)=2$  the associated functions 
$f : {\mathbb F}_{q} \rightarrow {\mathbb F}_{q}$ are called APN 
(Almost Perfectly Nonlinear). These functions
 have been extensively studied
as they offer good resistance against  differential attacks
(see \cite{Biham_Shamir}).
Among them, those which are APN over infinitely
many extensions of ${\mathbb F}_{q}$ have
attracted special attention.

In the opposite direction 
the third author proved  in  \cite{Felipe} that most polynomials $f \in {\mathbb F}_{q}[x]$ 
of degree $m\equiv 0$ or $3\pmod 4$ have differential uniformity 
equal to $m-1$ or $m-2$, the largest possible for polynomials of degree $m$. 
Precisely, he proved that for a given integer $m>4$
such that $m\equiv 0 \pmod 4$ 
(respectively $m \equiv 3 \pmod 4$),
if $\delta_0=m-2$
(respectively $\delta_0=m-1$) then
$$ 
\lim_{n\rightarrow \infty} 
\frac{
\sharp \{ f\in{\mathbb F}_{2^n}[x] \mid \deg(f)=m,\ \delta(f)=\delta_0 \} } 
{
\sharp 
\{
f\in{\mathbb F}_{2^n}[x]
 \mid 
\deg(f)=m
\} } = 1. $$

The first two authors extended this result
to the second order differential uniformity in \cite{YvesFabien}.

The following conjecture is also stated in  \cite{Felipe}: \\
\begin{conjecture}\label{conjecture}
For a given integer $m>4$, 
there exists
$\varepsilon_m >0$ such that for all sufficiently large $n$,
 if $f$ is a polynomial of degree $m$ over 
$ {\mathbb F}_{2^n}$, for at least 
$\varepsilon_m 2^{2n}$ values of $(\alpha, \beta) \in   {\mathbb F}^{\ast}_{2^n}\times{\mathbb F}_{2^n}$
we have
$ \sharp \{x\in{\mathbb F}_{q} \mid f(x+ \alpha) + f(x)=\beta\} = \delta(f)$.
\end{conjecture}

Moreover, it was proved in \cite{Felipe} that all polynomials $f$ of degree 7 have maximal differential uniformity (that is here $\delta(f)=6$) if $n$ is large enough.

The aim of this paper is to exhibit an infinite 
set $\mathcal M$ (defined below) of 
integers $m$ such that
every polynomial $f\in{\mathbb F}_{2^n}[x]$ of degree $m$ has maximal differential uniformity  if $n$ is large enough, that is $\delta(f)$ is equal to the degree of $D_{\alpha}f(x)=f(x+\alpha)+f(x)$, the derivative of $f$ with respect to $\alpha$.
We stress that, for $m \in \mathcal M$, our results are much stronger than those of
\cite{Felipe} as we prove maximality of differential uniformity for all
polynomials of degree $m$, as opposed to most of them.

\begin{dfn}(Definition \ref{etlabete} and
Proposition \ref{proposition:condition_xm})
We  denote by $\mathcal M$ the set of the odd integers $m$
such that
the unique polynomial $g$ 
satisfying 
$g \left( x(x+1) \right) = D_{1}(x^m)$
has distinct critical values.

We have that
$m$ belongs to $\mathcal{M}$ if and only if
for any $\zeta_1$ and $\zeta_2$ 
in $\overline{\mathbb{F}}_2 \setminus \{ 1 \}$,
the equalities 
$\zeta_1^{m-1}  = \zeta_2^{m-1} = 
\left( \frac{1+\zeta_1}{1+\zeta_2}\right)^{m-1} = 1
$ imply $\zeta_1=\zeta_2$ or $\zeta_1=\zeta_2^{-1}$.
\end{dfn}

Now we can state our main results. 

\begin{thm}(Theorem \ref{Principal_7_mod_8} and Theorem \ref{proof_conjecture})
Let $m\in{\mathcal M}$ such that $m\equiv 7\pmod 8$.
Then for $n$ sufficiently large,
for all polynomials $f\in{\mathbb F}_{2^n}[x]$
of degree $m$ we have  $\delta(f)=m-1$.
Furthermore, Conjecture 1.1. is true for such integers $m$.
\end{thm}

For example, we will prove that the previous theorem applies for the 
integers $m\in\{7, 23, 39, 47, 55, 79, 87, 95, 111, 119, 135, 143, 159, 167, 175, 191, 199\}$ (see Example \ref{list_integers_m}).
We also provide  explicit infinite families of such integers $m$, namely the integers
$m=2\ell^{2k+1}+1$ for $k\geqslant 0$ and 
  $\ell \in \{ 3,11,19,23,43,47, \break 59,67,71,79,83,  103, 107,131,139,151,163,167,179,191,199 
 \}$ (see Corollary \ref{First_family}).

When $m$ is congruent to $3$ modulo $8$, we also obtain some results but
we have conditions on the parity of $n$ or 
we have to remove some polynomials.

\begin{thm}(Theorem \ref{Principal_3_mod_8})

Let  $m\in{\mathcal M}$ such that  $m\geqslant 7$ and $m\equiv 3\pmod 8$.

\begin{enumerate}[label=(\roman*)]
\item For $n$ even and sufficiently large and for all polynomials $f\in{\mathbb F}_{2^n}[x]$
of degree $m$ we have $\delta(f)=m-1$.

\item For $n$ sufficiently large and for all polynomials $f=\sum_{i=0}^{m} a_{m-i} x^i$ 
in ${\mathbb F}_{2^n}[x]$ of degree $m$ such that $a_1^2+a_0a_2 \neq 0$, we have $\delta(f)=m-1$.
\end{enumerate}

\end{thm}

We  also provide infinite families of integers $m\equiv 3\pmod 8$ for which the previous theorem applies, 
namely the integers 
$m=2\ell^{k}+1$ for $k\geqslant 1$ and
  $\ell \in \{ 17,41,97,113,137,193
 \}$ and the integers 
$m=2\ell^{2k}+1$ for $k\geqslant 1$ and
  $\ell \in \{
 23,47,71,79,103,151,167,191,199
 \}$ (see  Corollary \ref{Second_family}).

\bigskip

Let us explain the strategy of the proofs of the above theorems
 which 
has 
 important similarities to that of \cite{Felipe} and \cite{YvesFabien}.
 For simplicity we consider in this sketch
 the case where $m$ is congruent to $7$ modulo $8$.

If $f \in {\mathbb F}_{q}[x]$ is a polynomial of degree $m$
 and if  $\alpha \in{\mathbb F}_{q}^{\ast}$, we introduce  
the unique polynomial $L_{\alpha}f$ 
of degree $d=(m-1)/2$
such that
 $L_{\alpha}f \left( x(x+\alpha)\right)=D_{\alpha}f(x)$
 (see Proposition \ref{polynomeg}).  
We consider  the splitting field  $F$ of the
polynomial $L_{\alpha}f(x)-t$ over the field ${\mathbb F}_q(t)$ with $t$ transcendental over  ${\mathbb F}_q$
and set  ${\mathbb F}_q^{F}$ be the algebraic closure of  $\mathbb F_q$ in $F$.
The Galois groups $G=\Gal(F/{\mathbb F}_q(t))$ and $\overline G=\Gal(F/{\mathbb F}_q^{F}(t))$ are respectively  the arithmetic and geometric monodromy groups of $L_{\alpha}f$.

If $u_0,\ldots,u_{d-1}$ are the roots of $L_{\alpha}f(x)=t$, then 
we will denote by $x_i$
a root of
$x^2+ \alpha x=u_i$.
So the $2d$ elements $x_0,x_0+\alpha,\ldots,x_{d-1},x_{d-1}+ \alpha$
are the solutions of
$D_{\alpha}f(x)=t$.
Thus we consider $\Omega={\mathbb F}_q(x_0,\ldots, x_{d-1})$ 
 the compositum of the fields $F(x_i)$ and  ${\mathbb F}_q^{\Omega}$ 
the algebraic closure of ${\mathbb F}_q$ in $\Omega$.
We set  also $\Gamma=\Gal(\Omega/F)$ and $\overline\Gamma=\Gal(\Omega/F{\mathbb F}_q^{\Omega})$.
Then we have the following diagram:

$$
\begin{tikzpicture}[node distance=1.5cm]
 
 \node (Fqdet)          {${\mathbb F}_q(t)$};
 \node (vide) [above of=Fqdet] {};
\node (F)   [above of=vide]   {$F={\mathbb F}_q(u_0,\ldots, u_{d-1})$};
 \node (FqFdet) [right of=vide] {${\mathbb F}_q^{F}(t)$};
 
  \node (vide2) [above of=F] {};
\node (Omega)   [above of=vide2]   {$\Omega={\mathbb F}_q(x_0,\ldots, x_{d-1})$};
 \node (FqOmegadeF) [right of=vide2] {$F{\mathbb F}_q^{\Omega}$};
  
  \draw (Fqdet) to node[left, midway,scale=0.9]  {$G$} (F);
  \draw (Fqdet)--(FqFdet);
  \draw (FqFdet) to node[right, midway,scale=0.9]  {\ $\overline G$} (F);
  
  \draw (F) to node[left, midway,scale=0.9]  {$\Gamma$} (Omega);
  \draw (F)--(FqOmegadeF);
  \draw (FqOmegadeF) to node[right, midway,scale=0.9]  {\ $\overline \Gamma$} (Omega);
  \end{tikzpicture}
$$
When the integer $m$ belongs to $\mathcal{M}$ and is congruent to $7$ modulo $8$
 we prove that for $n$ sufficiently large and
 for any polynomial $f \in {\mathbb F}_{2^n}[x]$ of degree $m$, 
  there exists  $\alpha$ in ${\mathbb F}_{2^n}^{\ast}$ such that:
 \begin{enumerate}
 \item  $L_{\alpha}f$ is Morse
\item  the equation $x^2+ \alpha x = \frac{b_1}{b_0}$ has a solution in $\mathbb{F}_{2^n}$.
\end{enumerate}

Now, condition (1) implies by Proposition \ref{Monodromy} that the extension $F/{\mathbb F}_{q}(t)$ is regular.
Condition (1) and (2) imply by Proposition \ref{Galois_second_floor} that the extension $\Omega/F$ is regular.
 It enables us to apply Chebotarev density theorem (see Proposition \ref{application_Chebotarev}) to obtain, for $n$ sufficiently large depending only on $m$, the existence of $\beta\in {\mathbb F}_{2^n}$ such that the polynomial 
$D_{\alpha}f(x)+\beta$ 
splits in $ \mathbb{F}_{2^n}[x]$ with no repeated factors. The differential uniformity of $f$ is thus equal to the degree of $D_{\alpha}f$.

\bigskip
The paper is organized as follows.
 Section \ref{section_nabla}   is devoted to the study of the  operator $L_{\alpha}$.
Section \ref{Morse} provides a detailed exposition  of Morse polynomials in even characteristic. According to the appendix by Geyer in \cite{JardenRazon}, 
Morse polynomials in this context are polynomials of odd degree satisfying two conditions: their critical points are non degenerate and their critical values are distinct. 
The first condition leads to the study of the number of $\alpha$
 such that the resultant of the derivative $(L_{\alpha}f)'$ with the second Hasse-Schmidt derivative
  $(L_{\alpha}f)^{[2]}$ does not vanish (Proposition \ref{proposition:condition_a}).
We give upper bounds for the number of exceptions in terms of $m$.

By contrast, we need additional requirements on $m$ to guarantee 
that for enough $\alpha$ the polynomial $L_{\alpha}f$ has distinct critical values
(see Proposition \ref{proposition:condition_b}).
Precisely, we will make the assumption that $L_{1}(x^m)$ has distinct critical values, 
this is that $m$ belongs to $\mathcal{M}$ (Definition \ref{etlabete}).
We complete Section 3 by exhibiting some families of infinitely
many integers belonging to $\mathcal M$.

Section \ref{Regular} is devoted to the study of the Galois groups $G$, $\overline G$, $\Gamma$ and $\overline\Gamma$. We prove in Proposition \ref{Galois_second_floor} that if the equation $x^2+ \alpha x = \frac{b_1}{b_0}$ has a solution in $\mathbb{F}_{2^n}$ i.e. if
${\Tr}_{{\mathbb F}_{2^n}/{\mathbb F}_2}  \left( \frac{b_1}{b_0\alpha^2} \right)=0$ then the extension $\Omega/F$ is regular.
The different expressions of $b_1/b_0$ 
we have obtained
in Lemma \ref{lemma:b1overb0},
depending on the congruence of $m$ modulo $8$,
induce differences in the treatment.

Section \ref{Main} deals with   the Chebotarev density theorem and contains the statements and the proofs of the main results.
\bigskip

Let us stress 
the main difference  between  the common approach of  \cite{Felipe} and \cite{YvesFabien} and   the approach of the present paper. For simplicity, we consider again that $m\equiv 7\pmod 8$. In  \cite{Felipe} and \cite{YvesFabien}, one of the key steps is to fix $\alpha_1,\ldots,\alpha_k$ in ${\mathbb F}_{2^n}$ and to obtain a lower bound  depending on $n$ for the number of polynomials $f$ in ${\mathbb F}_{2^n}[x]$ such that 
at least one of the $L_{\alpha_i}f$ is Morse.
By contrast, we prove here that for $n$ sufficiently large and for any polynomial $f$ of degree $m$ in ${\mathbb F}_{2^n}[x]$ there exists $\alpha$ such that $L_{\alpha}f$ is Morse.

\section{The associated polynomial $L_{\alpha}f$}\label{section_nabla}
Let $f \in {\mathbb F}_q[x]$ be a polynomial of degree $m \geqslant 7$  (the cases where $m<7$ are handled in \cite{Felipe})  and $\alpha\in{\mathbb F}_q^{\ast}$.
The derivative of a polynomial $f\in{\mathbb F}_q[x]$
along $\alpha$ is defined by:
$$D_{\alpha}f(x)=f(x)+f(x+\alpha).$$

If we set $f=\sum_{k=0}^{m} a_{m-k} x^k$,
a straightforward computation gives that 
$D_{\alpha}f= \sum_{k=0}^{m} c_{m-k} x^k$
where $c_k=a_{k}+\sum_{i=m-k}^{m} a_{m-i} \binom{i}{m-k} \alpha^{i-m+k}$.
As we work over an even characteristic field,
we have
$ c_0=a_0 + a_0=0$,
$c_1=m \alpha a_0$ and $ c_2=(m-1) \alpha a_1 + \binom{m}{2} \alpha^2 a_0.$
We deduce the following proposition.

\begin{proposition}\label{degree_of_g}
Let $f\in{\mathbb F}_q[x]$ be a polynomial of degree $m$. 
If $m$ is odd then  the degree of $D_{\alpha}f$ is $m-1$.
If $m$ is even then the degree of $D_{\alpha}f$ is less than or equal to $m-2$, 
and equal to $m-2$ if and only if $a_1+a_0 \alpha \binom{m}{2} \neq 0$.
\end{proposition}

In the whole paper, we will associate to any integer $m$ the following integer $d$.

\begin{definition}\label{def:d}
 Let $m$ be an integer. Suppressing in our notation the dependence on $m$, we set  $d=(m-1)/2$ if $m$ is odd and $d=(m-2)/2$ is $m$ is even.
 \end{definition}

\subsection{Existence of $L_{\alpha}f$}

\begin{proposition}\label{polynomeg}
 Let $\alpha\in{\mathbb F}_q^{\ast}$ and
let $f\in{\mathbb F}_q[x]$ be a polynomial of degree $m$. Then there exists a 
unique polynomial $g\in{\mathbb F}_q[x]$ 
of degree less than or equal to
 $d$  such that
$$D_{\alpha}f(x)=g(x(x+\alpha)).$$

Furthermore, the map $L_{\alpha}:f\longmapsto g$ is linear and 
its restriction   to the subspace of polynomials
of degree at most $m$  is surjective onto the subspace
of polynomials of degree at most $d$.
\end{proposition}

\begin{proof}
The proof is similar to that of Proposition 2.2. of \cite{YvesFabien} dealing with the set $\Lambda_{k}$  of roots
of multiplicity $k$ of $D_{\alpha}f$ and noticing that $x \mapsto x + \alpha$ is an involution
of each set $\Lambda_{k}$.
The surjectivity of $L_{\alpha}$ follows from the fact that 
the kernel of the  restriction of $L_{\alpha}$ to 
the space of  polynomials of degree at most $m$  is  
the subspace of  polynomials $g \left( x(x+ \alpha) \right)$
where $g \in {\mathbb F}_q[x]$ has degree at most $[m/2]$
(see Lemma 2.3. of \cite{YvesFabien}).
\end{proof}


\subsection{The coefficients $b_i$ of $L_{\alpha}f$}
\
\smallskip

Let $f=\sum_{i=0}^{m} a_{m-i} x^i\in{\mathbb F}_q[x]$ be a polynomial of degree $m$ and $L_{\alpha}f=\sum_{i=0}^{d} b_{d-i} x^i$ be the associated polynomial of degree  $d$ when $m$ is odd and of degree less than or equal to $d$ otherwise (see Proposition \ref{degree_of_g}).
To obtain information on the coefficients $b_i$, one
can consider the triangular linear system
with coefficients $1$ on the diagonal
arising when identifying 
the coefficients
of $x^{2d}, x^{2d-2}, \ldots,x^2,x^0$
in $g \left( x(x+\alpha ) \right)$ and in $D_{\alpha}f$.
Note that this approach
proves again the unicity of $g$
claimed in Proposition \ref{polynomeg}.

More precisely, 
a necessary condition for the term $b_s x^t$
to appear in $g \left( x(x+\alpha) \right)$ is that
$d-t \leqslant s \leqslant d-t/2$.
In this case, it appears with the
coefficient
$\binom{d-s}{t-d+s} \alpha^{2(d-s)-t}$.
So for each integer $k$ between $0$ and $d$,
identifying
the coefficient of $x^{2(d-k)}$ in $g \left( x(x+\alpha) \right) $
and in $D_{\alpha}f(x)$ gives

\begin{equation}\label{relations}
 \sum_{
s=\Max 
\{0 , 2k-d\}
} ^{k}
\binom{d-s}{2k-2s} \alpha^{2k-2s} b_s = 
\sum_{i =2d-2k+1}^{m} 
\binom{i}{2d-2k} \alpha^{i-2d+2k} a_{m-i}.
\end{equation}

We consider the polynomial ring  $\mathbb{F}_2 [\alpha, a_0, \ldots, a_m ]$
where $\alpha, a_0,\ldots, a_m$ are indeterminates with the degree $w$  such that $w(\alpha)=1$ and $w(a_j)=j$.   It means that the monomial $\alpha^{d_{\alpha}}a_0^{d_0}a_1^{d_1}a_2^{d_2}\ldots a_m^{d_m}$ has degree $d_{\alpha}+d_1+2d_2+\cdots +md_m$.
Then using the triangular system obtained from  (\ref{relations}) and an induction on $k$ 
prove the following homogeneity result.
\begin{lemma}\label{coeff_b_i}
For all integers $i$ such that $0 \leqslant i \leqslant d$
we have $b_i \in \mathbb{F}_2 [\alpha, a_0, \ldots, a_m ]$
which is an homogeneous polynomial of degree $2i+1$ 
if $m$ is odd and
of degree $2i+2$ if $m$ is even,
when considering
the degree $w$ such that $w(\alpha)=1$ and $w(a_j)=j$.
\end{lemma}
The relations (\ref{relations}) also provide expressions of the first coefficients $b_0, b_1,\ldots$ of $L_{\alpha}f$
depending on the congruence class of $m$ modulo $8$, as made explicit
in the next lemma which will be needed in the proof of Theorem \ref{Principal_7_mod_8}. 
Note that formulas for $b_1/b_0$ 
appeared in \cite{Felipe} as well, but the last two had misprints.

\begin{lemma}\label{lemma:b1overb0}
Let $m$ be an integer. If $m\equiv 0\pmod 4$ then $b_0=a_1\alpha$ and if $m\equiv 3\pmod 4$ then $b_0=a_0\alpha$. Moreover, we have the following expressions of $b_1/b_0$ depending on the congruence of $m$:

\begin{displaymath}
\begin{array}{cc} 
\hline\\[-5pt]
m \pmod 8 & b_1/b_0    \\[5pt]
\hline\\
3          & \alpha^2 +\frac{a_1 \alpha + a_2}{a_0}                                          \\[5pt]
7           & \frac{a_1 \alpha + a_2}{a_0}                                        \\[5pt]
0         & \frac{a_2\alpha+a_3}{a_1}                                     \\[5pt]
4          & \alpha^2+\frac{a_0 \alpha^3+a_2\alpha+a_3}{a_1}                                     \\[7pt]
\hline
\end{array}
\label{table:definition_of_b_1/b_0}
\end{displaymath}

\end{lemma}


\section{For almost every $\alpha$ the polynomial $L_{\alpha}f$ is Morse}\label{Morse}

We will focus now on polynomials $f$ of degree $m\equiv 3\pmod{4}$ and thus, for nonzero $\alpha$,  on polynomials $L_{\alpha}f$ of odd degree $d=(m-1)/2$.

\subsection{Morse polynomials in even characteristic}
We consider the following notion of Morse polynomial given in all characteristic by Geyer in an appendix to the paper \cite{JardenRazon}.

\begin{definition}\label{Definition_Morse}
Let $K$ be a field of characteristic $p \geqslant 0$.
We say that a polynomial $g$ over $K$ is  Morse 
if the three following conditions hold:
\begin{enumerate}[label=(\alph*)]
\item the critical points of $g$, i.e the zeroes of $g'$, are non degenerate, 
\item the critical values of $g$ are distinct, i.e. 
$g'(\tau)=g'(\eta)=0$ 
and $g(\tau)=g(\eta)$ imply $\tau=\eta$,
\item if $p>0$, then the degree of $g$ is not divisible by $p$.
\end{enumerate}
\end{definition}
These conditions are chosen such that $g$
corresponds to a covering with maximum Galois group, 
that is $\Gal \left( g(t)-x , K(x) \right)$  is the symmetric group $\Sgoth_d$ where $d$ is the degree of $g$
 (see Proposition 4.2 in \cite{JardenRazon}).
In the case where $p>0$, 
the loci of non-Morse polynomials is described in 
the same appendix.

Let us sum up the situation 
in the case where $p=2$.
In this case
one has to introduce the Hasse-Schmidt derivative $g^{[2]}$ which is defined by 
the equality 
$g(t+u) \equiv g(t)+g'(t)u + g^{[2]}(t)u^2 \pmod{u^3}$
where $u$ and $t$ are independent variables.
If $g=\sum_{i=0}^db_{d-i}x^i$ is a degree $d$ polynomial of ${\mathbb F}_q[x]$ with $q$ a power of 2, then
the condition (a) above is fulfilled if and only if $g'$ and $g^{[2]}$ have
no common roots, that is if and only if the resultant
$$R:=\Res(g',g^{[2]}) \in \mathbb{F}_2[b_0,\ldots,b_d]$$
 does not
vanish. 
And the condition (b) above is fulfilled if and only if 
$$\Pi(g):=\prod_{i \neq  j} \left( g(\tau_i) - g(\tau_j) \right)$$
 does not vanish,
where $\tau_1, \ldots, 
\tau_{\left[ \frac{d-1}{2} \right] }$ 
are the (double) roots of $g'$.
Using the theorem on symmetric functions, one can obtain an expression of 
$\Pi(g)$ depending on the coefficients $b_0, \ldots , b_d$ of $g$.

In order to calculate the second order Hasse-Schmidt derivative, 
we will make use of the following Lucas theorem about binomial coefficients (see for instance the introduction of \cite{Granville}).
 For
$p$  a prime number, write $m=m_0+m_1p+m_2p^2+\cdots +m_rp^r$ and $k=k_0+k_1p+k_2p^2+\cdots +k_rp^r$ in base $p$.
Then we have
$\displaystyle{
\binom{m}{k}\equiv \binom{m_0}{k_0}\binom{m_1}{k_1}\cdots \binom{m_r}{k_r} \pmod p.
}$

\subsection{The condition (a)}

In order to bound the number of $\alpha$ such that
the critical values of $L_{\alpha}f$ are non degenerate, we study in 
this subsection
  $\Res((L_{\alpha}f)',(L_{\alpha}f)^{[2]}) \in 
\mathbb{F}_2[a_0,\ldots,a_m]$.

We will need three lemmas to succeed in doing so.
Lemma \ref{lemma:from_L_to_D} enables us to study 
$\tilde{R}:=\Res \left( (D_{\alpha}f)' , (D_{\alpha}f)^{[2]} \right)$
rather than $\Res \left( (L_{\alpha}f)' , (L_{\alpha}f)^{[2]} \right)$. Then
Lemma \ref{lemma:R_is_homogeneous} gives a result about the homogeneity and the degree 
of this polynomial if it is nonzero.
To prove its non nullity 
we evaluate it in $a_0=1,a_1=\cdots=a_m=0$ which amounts to
determining in Lemma \ref{lemma:R_x_m} if the polynomial $x^m$ has non degenerate critical points.

\begin{proposition}\label{proposition:condition_a}
Let $m\geqslant 7$ such that $m \equiv 3 \pmod 4$
and 
let $f(x)=\sum_{k=0}^{m} a_{m-k}x^k$ be a polynomial
of $\mathbb{F}_q[x]$
of degree $m$.
Then the critical points of $L_{\alpha}f$ are non degenerate except for at most $m(m-3)$ values
of $\alpha \in \overline{\mathbb F}_2$.
\end{proposition}

\begin{lemma}\label{lemma:from_L_to_D}
Let $f\in \mathbb{F}_q[x]$ be a polynomial.
For all $\alpha \in \mathbb{F}_q^{\ast}$ 
the polynomials $(L_{\alpha}f)'$ and $(L_{\alpha}f)^{[2]}$
have a common root in $\overline{\mathbb F}_2$ if and only if
the polynomials $(D_{\alpha}f)'$ and $(D_{\alpha}f)^{[2]}$ 
have a common root in $\overline{\mathbb F}_2$.
\end{lemma}

\begin{proof} Since  $D_{\alpha}f=L_{\alpha}f\circ T_{\alpha}$ where $T_{\alpha}(x):=x(x+ \alpha)$, we can prove
the two following equalities:
$$(D_{\alpha}f)'=\alpha (L_{\alpha}f)'\circ T_{\alpha}$$
and
$$(D_{\alpha}f)^{[2]}=(L_{\alpha}f\circ T_{\alpha})^{[2]}=(L_{\alpha}f)'\circ T_{\alpha} + \alpha^2(L_{\alpha}f)^{[2]}\circ T_{\alpha}.$$
The result follows.
\end{proof}

\begin{lemma}\label{lemma:R_is_homogeneous}
Let $m\geqslant 7$ 
such that $m \equiv 3 \pmod 4$
and let $f=\sum_{k=0}^{m} a_{m-k}x^k$
in $\mathbb{F}_2[a_0,\ldots,a_m] [x]$.
Consider the degree $w$ defined by $w(\alpha)=1$ and
$w(a_i)=i$ for any $i$ and consider also the degree $\tilde w$ defined by
$\tilde{w}(\alpha)=0$ and
$\tilde{w}(a_i)=1$.

Then 
the resultant $\Res \left( (D_{\alpha}f)' , (D_{\alpha}f)^{[2]} \right)$ 
in the variable $x$, if it is nonzero,
is an homogeneous polynomial of ${\mathbb{F}_2}[a_0,\ldots,a_m,\alpha]$
of degree $m(m-3)$  when considering 
the degree $w$  and is an homogeneous polynomial of degree $2(m-3)$ 
when considering 
the degree $\tilde{w}$.
\end{lemma}

\begin{proof}
As $f(x)=\sum_{k=0}^{m} a_{m-k}x^k$ and 
$f(x+\alpha)=\sum_{k=0}^{m} a_{m-k}(x+\alpha)^k$, these two polynomials
are homogeneous of degree $m$ for the degree $w$ such that
$w(\alpha)=1$, 
$w(a_i)=i$ and $w(x)=1$.
 It follows that $(D_{\alpha} f) '$ and $(D_{\alpha} f) ^{[2]}$
 are  homogeneous of degree respectively $m-1$ and $m-2$ for 
 the  degree $w$. 
Using the formulae of $D_{\alpha}f$ given in Section \ref{section_nabla}, we have:
$$D_{\alpha}f(x)=\alpha a_0 x^{m-1}+a_0\alpha^2x^{m-2}
+\left( a_0\alpha^3+a_1\alpha^2+a_2\alpha \right)x^{m-3}+\cdots$$

The polynomial
 $(D_{\alpha} f) '$ has degree $m-3$
 in the variable $x$ since $m$ is odd and its leading coefficient  is $a_0\alpha^2$.
 The polynomial $(D_{\alpha} f) ^{[2]}$ has also degree $m-3$
 in the variable $x$ since it can be shown that  $(x^k)^{[2]}=\binom{k}{2}x^{k-2}$ using  the binomial theorem, the above Lucas theorem and the congruence of $m$. Its leading coefficient is 
  $a_0\alpha$.
  
Thus we can set
 $(D_{\alpha} f) ' = \sum_{i=0}^{m-3} d_i x^{m-3-i}$ and
 $(D_{\alpha} f)^{[2]} = \sum_{i=0}^{m-3} e_i x^{m-3-i}$
 where $d_i,e_i \in \mathbb{F}_2[a_0,\ldots,a_m,\alpha]$ 
 are such that $w(d_i)=i+2$ and $w(e_i)=i+1$.
 Thus the resultant  $\Res \left( (D_{\alpha}f)' , (D_{\alpha}f)^{[2]} \right) $
 in the variable $x$, if it is nonzero,
 is an homogeneous polynomial of $\mathbb{F}_2[a_0,\ldots,a_m,\alpha]$ 
 of degree $m(m-3)$ for the degree $w$.
 For the second homogeneity result claimed, 
 note that this resultant is a sum of $2(m-3)$ products 
 of the coefficients $d_i$ and $e_i$, and each one of them
is a  linear combination in the $a_0,\ldots,a_m$.
\end{proof}

\begin{lemma}\label{lemma:R_x_m}
Let $m\geqslant 7$ such that $m \equiv 3 \pmod 4$
and let $f=x^m$.
For all $\alpha \in \mathbb{F}_q^{\ast}$ the critical points of $L_{\alpha}f$ are non degenerate.
\end{lemma}
\begin{proof}
Using Lemma \ref{lemma:from_L_to_D} we look for the common roots of 
$(D_{\alpha}f)'$ and $(D_{\alpha}f)^{[2]}$. 
We compute 
$(D_{\alpha}f)'= 
(x+ \alpha)^{m-1}+x^{m-1}$ and 
$(D_{\alpha}f)^{[2]} = 
(x+ \alpha)^{m-2}+x^{m-2}$.
Hence, if $\omega \in 
\overline{\mathbb F}_2$ 
was
a common root of 
$(D_{\alpha}f)'$ and $(D_{\alpha}f)^{[2]}$ then
 we would have 
$\left( (\omega+\alpha)  / \omega \right)^{m-1}=\left( (\omega+\alpha) / \omega \right)^{m-2}=1$, 
and so
$\alpha=0$.
\end{proof}

Now we are enable to prove Proposition \ref{proposition:condition_a}.
\begin{proof}
Lemma \ref{lemma:from_L_to_D} enables us to study 
$\tilde{R}:=\Res \left( (D_{\alpha}f)' , (D_{\alpha}f)^{[2]} \right)$
rather than $\Res \left( (L_{\alpha}f)' , (L_{\alpha}f)^{[2]} \right)$.
Using the homogeneity results given by Lemma \ref{lemma:R_is_homogeneous}
we know that
there is at most one term in $\tilde{R}$
of degree at least $m(m-3)$ in $\alpha$, precisely
$a_0^{2(m-3)} \alpha^{m(m-3)}$.
We study whether this term appears or not.

By Lemma \ref{lemma:R_x_m}, for nonzero $\alpha$ the critical points of $L_{\alpha}(x^m)$ 
are non degenerate, so $\tilde{R}(a_0=1,a_1=0,\ldots,a_m=0,\alpha=1) \neq 0 $ 
and
 this term does appear.
Choosing a polynomial 
 $f\in \mathbb{F}_q[x]$ of degree $m$
 amounts to choosing
 coefficients $a_0 ,\ldots,a_m$ in ${\mathbb F}_q$ with $a_0 \neq 0$. Thus we can consider $\tilde{R}$  as
a nonzero polynomial in $\alpha$ of degree $m(m-3)$ 
which
has at most 
$m(m-3)$ roots.
\end{proof}

\subsection{The condition (b)}
We use a similar strategy to prove that
for almost every choice of $\alpha$ the polynomial
$L_{\alpha}f$ has distinct critical values: we use an homogeneity
result and we study the case of  $L_{\alpha}(x^m)$. 
As it is a key point in our approach, we give equivalent 
conditions for $L_{\alpha}(x^m)$ to have distinct critical values.
Recall that we work with $m\equiv 3\pmod{4}$ and that we set $d=(m-1)/2$.

\begin{proposition}\label{proposition:condition_b}
Let $m$ be an integer such that $m\geqslant 7$ and $m \equiv 3 \pmod 4$.

\begin{enumerate}[label=(\roman*)]
\item If there exists $\alpha \in \overline{\mathbb F}_2^{\ast}$ such that  $L_{\alpha}(x^m)$ has distinct critical values
then it holds true for any $\alpha \in \overline{\mathbb F}_2^{\ast}$.

\item Suppose that for any $\alpha \in \overline{\mathbb F}_2^{\ast}$ (or equivalently for  $\alpha=1$) the polynomial $L_{\alpha}(x^m)$ has distinct critical values.
Let $f\in \mathbb{F}_q[x]$ be a polynomial of degree $m$.
Then $L_{\alpha}f$ has distinct critical values
except for at most $(5m-1)(m-3)(m-7)/64$ values
of $\alpha \in \overline{\mathbb F}_2$.
\end{enumerate}
\end{proposition}

\begin{proof} Let $\alpha \in \overline{\mathbb F}_2^{\ast}$ such that  $L_{\alpha}(x^m)$ has distinct critical values. Now let $\alpha' \in \overline{\mathbb F}_2^{\ast}$ and let us show that $L_{\alpha'}(x^m)$ has distinct critical values. We use the characterization given by Lemma \ref{lemma:moving_from_L_to_D}:  suppose that 
$(\tau,\eta) \in  \left( \overline{\mathbb{F}}_2 \right)^2$ are such that
\begin{equation}\label{critic1}
\tau^{m-1}+(\tau+\alpha')^{m-1}=\eta^{m-1}+(\eta+\alpha')^{m-1}=0
\end{equation}
and
\begin{equation}\label{critic2}
\tau^{m}+(\tau+\alpha')^{m}=\eta^{m}+(\eta+\alpha')^{m}.
\end{equation}

Mutliply Equation (\ref{critic1}) by $\left(\frac{\alpha}{\alpha'}\right)^{m-1}$ and Equation (\ref{critic2}) by $\left(\frac{\alpha}{\alpha'}\right)^{m}$, we obtain that $\frac{\alpha}{\alpha'}\eta\in\{\frac{\alpha}{\alpha'}\tau,\frac{\alpha}{\alpha'}\tau+\alpha\}$ i.e. $\eta\in\{\tau, \tau+\alpha'\}$ which gives the result.

To prove assertion $(ii)$ we follow the strategy of the proof of 
Proposition \ref{proposition:condition_a}.
Consider 
$f=\sum_{i=0}^{m} a_{m-i}x^i
\in \mathbb{F}_2[a_0,\ldots,a_m] [x]$ 
and $L_{\alpha}f=\sum_{i=0}^db_{d-i}x^{i} \in \mathbb{F}_2[b_0,\ldots,b_d,\alpha] [x]$. 
By Lemma \ref{lemma:Pi_is_homogeneous}, when setting $N= d \binom{(d-1)/2}{2}$ we can see 
  $b_0^N \times\Pi  \left(L_{\alpha} f\right)$ as a polynomial of
${\mathbb{F}_2}[a_0,\ldots,a_m,\alpha]$.
Now we use the 
homogeneity result of Lemma \ref{lemma:Pi_is_homogeneous}
to know that 
this last polynomial has at most one term
of degree at least
$(5d+2)\binom{(d-1)/2}{2}$ in $\alpha$ .
Precisely, this term is possibly the term
$a_0^{(d+2)\binom{(d-1)/2}{2}} \alpha^{ (5d+2)\binom{(d-1)/2}{2}}$.

In order to know if this term appears or not, we evaluate this polynomial
at $a_0=1$ and $a_i=0$ for all $i>0$ which amounts to determine if 
 the polynomial $L_{\alpha}(x^m)$ has distinct critical values, which is true
 by hypothesis.
 Now fix a polynomial $f\in \mathbb{F}_q[x]$ of degree $m$ and see  $b_0^N \times\Pi (L_{\alpha}f)$
 as a polynomial of $\mathbb{F}_2 [\alpha]$.
 So we know its degree and 
thus $L_{\alpha}f$ has distinct critical values
except for at most $(5d+2) \binom{(d-1)/2}{2}$ values 
of $\alpha \in \overline{\mathbb F}_2$.
Then we conclude using the relation between $m$ and $d$.
\end{proof}

The following lemma gives a condition on $D_{\alpha}f$ for $L_{\alpha}f$ to have distinct critical values.

\begin{lemma}\label{lemma:moving_from_L_to_D}
Let $f\in \mathbb{F}_q[x]$.
For all $\alpha \in \mathbb{F}_q^{\ast}$ 
the polynomial $L_{\alpha}f$ has distinct critical values if and only if
for all $(\tau,\eta) \in  \left( \overline{\mathbb{F}}_2 \right)^2$, 
$(D_{\alpha} f)'(\tau)=(D_{\alpha} f)'(\eta)=0$
and  $D_{\alpha} f(\tau)=D_{\alpha} f(\eta)$ imply
$\tau=\eta$ or $\tau=\eta + \alpha$.
\end{lemma}

\begin{proof}
We have $L_{\alpha} f \circ T_{\alpha} = D_{\alpha} f$, so
$(D_{\alpha} f)' = \alpha \left( L_{\alpha}f \right)  ' \circ T_{\alpha}$
where $T_{\alpha}(x)=x(x+ \alpha)$.
The result follows noticing that $T_{\alpha} (\tau) = T_{\alpha} (\eta) $
if and only if $\tau \in \{ \eta, \eta+ \alpha\}$.
\end{proof}

\begin{lemma}\label{lemma:Pi_is_homogeneous}
Let $m\geqslant 7$ such that $m \equiv 3 \pmod 4$  and set $N= d \binom{(d-1)/2}{2}$.
We consider the polynomials 
$f=\sum_{k=0}^{m} a_{m-k}x^k
\in \mathbb{F}_2[a_0,\ldots,a_m] [x]$ 
and $L_{\alpha}f=\sum_{k=0}^db_{d-k}x^{k} \in \mathbb{F}_2[b_0,\ldots,b_d,\alpha] [x]$.
Then  $b_0^N \times\Pi  \left(L_{\alpha} f\right)$ is a polynomial of
${\mathbb{F}_2}[a_0,\ldots,a_m,\alpha]$ whose 
each term contains a product of $(d+2) \binom{(d-1)/2}{2}$ terms $a_i$.
This polynomial is also homogeneous  of degree $(5d+2) \binom{(d-1)/2}{2}$
 when considering 
the weight $w$ such that $w(\alpha)=1$
and 
$w(a_i)=i$.
\end{lemma}
\begin{proof}
We set $\tau_1, \ldots , \tau_{(d-1)/2}$ the double roots 
of the polynomial $(L_{\alpha}f)'$, and 
$\Pi  \left(L_{\alpha} f\right)=\prod_{i \neq  j} \left( L_{\alpha}f(\tau_i) - L_{\alpha}f(\tau_j) \right)$.
Then we have 
$$\Pi  \left(L_{\alpha} f\right)=\prod_{i <  j} \left( \sum_{k=0}^{d} b_{d-k}^2 
(\tau_i^{2k}+\tau_j^{2k}) \right).$$
So $\Pi  \left(L_{\alpha} f\right)$ is an homogeneous polynomial of degree $2d \binom{(d-1)/2}{2}$
when considering 
the weight $w$ such that $w(b_i)=i$ for all $i$
and 
$w(\tau_j)=1$ for all $j$.
We also have that 
 $\Pi  \left(L_{\alpha} f\right) \in \mathbb{F}_2[b_0,\ldots,b_d,\tau_1^2,\ldots,\tau_{(d-1)/2}^2] $,
 and each term of $\Pi  \left(L_{\alpha} f\right)$ contains a product of 
 exactly $\binom{(d-1)/2}{2}$ terms $b_i^2$.
 Moreover, using the invariance under the action of $\Sgoth_{(d-1)/2}$
 and the theorem of symmetric functions, we obtain that 
 $\Pi  \left(L_{\alpha} f\right) \in \mathbb{F}_2[b_0,\ldots,b_d, \sigma_1, \ldots, \sigma_{(d-1)/2}]$
 where $\sigma_1 = \sum \tau_i^2$, 
 $\sigma_2 = \sum_{i <j} \tau_i^2 \tau_j^2$,...
Using $(L_{\alpha}f)'=b_0 \prod_{i=1}^{(d-1)/2} \left( x^2+\tau_i^2 \right) $ it follows that
$\Pi  \left(L_{\alpha} f\right) \in \mathbb{F}_2[b_0,\ldots,b_d, \frac{b_2}{b_0},
 \frac{b_4}{b_0}, \ldots , \frac{b_{d-1}}{b_0} ]$.
The denominator is at worst $b_0^{N}$ 
(it happens if the $\tau_i$ 
are the only terms contributing to the degree, and if they
only give rise
to terms $b_2/b_0$).
We deduce that $b_0^N \times\Pi  \left(L_{\alpha} f\right)$ is a polynomial
in the $b_i$, 
and that each term is a product of $(d+2)\binom{(d-1)/2}{2}$ 
indeterminates $b_i$.
Furthermore, it is 
an homogeneous polynomial of degree $2d \binom{(d-1)/2}{2}$
when considering 
the weight $w$ such that $w(b_i)=i$ for all $i$.

By Lemma \ref{coeff_b_i}, 
$b_i$  is an homogeneous polynomial of
$\mathbb{F}_2[a_0,\ldots,a_m,\alpha]$ of degree $2i+1$  
 when considering the weight $w$ such that 
$w(a_i)=i$ and $w(\alpha)=1$.
We conclude that
 $b_0^N \times\Pi  \left(L_{\alpha} f\right)$ is an homogeneous polynomial
of degree $2 \times 2d \binom{(d-1)/2}{2} + (d+2)\binom{(d-1)/2}{2}$.
\end{proof}

\bigskip

Finally we reach the goal of this section: Proposition \ref{proposition:condition_a} and Proposition \ref{proposition:condition_b} enable us to bound the number of $\alpha$ such that $L_{\alpha}f$ is Morse.

\begin{theorem}\label{L_alpha_Morse}
Let $m \geqslant 7$ such that $m  \equiv 3 \pmod 4$ and such that  the polynomial $L_{1}(x^m)$ has distinct critical values.
Then
for all $f \in {\mathbb F}_{2^n}[x]$ of degree $m$ the number of elements $\alpha$ in 
${\mathbb F}_{2^n}^{\ast}$ such that $L_{\alpha}f$ is Morse is 
at least 
$2^n-1-\frac{1}{64}(m-3)(5m^2+28m+7)$.
\end{theorem}

\begin{proof}
Let $f \in {\mathbb F}_{2^n}[x]$ of degree $m$ and let $\alpha\in{\mathbb F}_{2^n}^{\ast}$. The polynomial $L_{\alpha}f$ is Morse if the three conditions (a), (b) and (c) of Definition \ref{Definition_Morse} hold. As $m  \equiv 3 \pmod 4$ the condition (c) is satisfied.
Indeed,  $D_{\alpha}f$ has degree $m-1$ by Proposition  \ref{degree_of_g} and thus $L_{\alpha}f$ has odd degree $(m-1)/2$. Moreover the condition (a) fails for at most $m(m-3)$ values of $\alpha$ by Proposition \ref{proposition:condition_a}. Furthermore the condition (b) fails for at most $(5m-1)(m-3)(m-7)/64$ values of $\alpha$ by Proposition \ref{proposition:condition_b}. Thus $L_{\alpha}f$ is not Morse for at most $m(m-3) + (5m-1)(m-3)(m-7)/64$ values of $\alpha$.
\end{proof}

\subsection{Conditions for  $L_1(x^m)$ to have distinct critical values.}\label{EnsembleM}

The condition (b) 
which is essential for the proofs
of our main results leads
by Proposition \ref{proposition:condition_b}  to study for which exponents $m$  the polynomial  $L_{\alpha} (x^m)$
has distinct critical values. By the first assertion of Proposition \ref{proposition:condition_b} we are reduced to consider 
the polynomial  $L_{1} (x^m)$.
Then it is natural to introduce the following set $\mathcal{M}$ and to 
look for practical characterizations.

\begin{definition}\label{etlabete}
Let $\mathcal M$ be the set of odd integers $m$ such that  the polynomial $L_{1}(x^m)$ has distinct critical values or equivalently such that  for any $\alpha \in \overline{\mathbb F}_2^{\ast}$ the polynomial $L_{\alpha}(x^m)$ has distinct critical values.
\end{definition}

Lemma \ref{lemma:moving_from_L_to_D} reduces the study of the critical values of 
$L_{\alpha}(x^m)$ 
to the study of equations involving
$D_{\alpha} \left( x^m \right)$ and $\left( D_{\alpha} \left( x^m \right) \right) '=x^{m-1}+(x+\alpha)^{m-1}$ for odd $m$.

The following proposition enables us to have a characterization 
of the elements of $\mathcal{M}$ in terms of 
roots of unity.

\begin{proposition}\label{proposition:condition_xm}
Let $m \geqslant 7$
be an odd integer. Whatever the choice of
$\alpha \in \overline{\mathbb F}_2^{\ast}$, the polynomial $L_{\alpha}(x^m)$ has distinct critical values
 if and only if the following condition is satisfied:
 
for $\zeta_1$ and $\zeta_2$ 
in $\overline{\mathbb{F}}_2 \setminus \{ 1 \}$,
the equalities 
$\zeta_1^{m-1}  = \zeta_2^{m-1} = 
\left( \frac{1+\zeta_1}{1+\zeta_2}\right)^{m-1} = 1
$ imply $\zeta_1=\zeta_2$ or $\zeta_1=\zeta_2^{-1}$.
\end{proposition}
\begin{proof}
We use Lemma \ref{lemma:moving_from_L_to_D} to relate 
with the equations of   Lemma \ref{lemma:condition_xm}.
With the expressions of $x_i$ and $x_j$ obtained,
 we notice that $x_i=x_j+ \alpha$ if and only if $\zeta_1 \zeta_2 =1$.
\end{proof}

\begin{lemma}\label{lemma:condition_xm}
Let $m \geqslant 7$ 
be an odd integer
and $\alpha \in \mathbb{F}_q^{\ast}$.
Two distinct
elements
$x_i$ and $x_j$ 
in $\overline{\mathbb{F}}_2$ satisfy
$$ x_i^{m-1}=(x_i + \alpha)^{m-1} ,  x_j^{m-1}=(x_j+ \alpha)^{m-1}\textrm{ and }$$
$$ x_i^{m}+(x_i + \alpha)^{m} =
x_j^{m}+(x_j+ \alpha)^{m} \ \ \ (\diamond) $$
if and only if
$x_i=\frac{\zeta_1 (1+\zeta_2)}{\zeta_1 + \zeta_2} \alpha $ and
$x_j=\frac{(1+\zeta_2)}{\zeta_1 + \zeta_2} \alpha $
where $\zeta_1$ and $\zeta_2 $ are two distinct elements
in $\overline{\mathbb{F}}_2 \setminus \{ 1 \}$
such that
$\zeta_1^{m-1}  = \zeta_2^{m-1} =
\left( \frac{1+\zeta_1}{1+\zeta_2}\right)^{m-1}=1$.
\end{lemma}

\begin{proof}
Suppose that $x_i$ and $x_j$ satisfy the 
first set of conditions above.
We
 notice that 
they cannot be $0$  neither $\alpha$, 
so we can set $\zeta_1=x_i/x_j$ and $\zeta_2=(x_i+\alpha)/(x_j+\alpha)$.
As $x_i \neq x_j$ we have $\zeta_1 \neq \zeta_2$.
Replacing $(x_i+\alpha)^{m-1}$ by $x_i^{m-1}$ 
and $(x_j+\alpha)^{m-1}$ by $x_j^{m-1}$ in $(\diamond)$ 
we obtain $\zeta_1^{m-1}=1$.
Replacing $x_i^{m-1}$ by $(x_i+\alpha)^{m-1}$   
and $x_j^{m-1}$ by $(x_j+\alpha)^{m-1}$ in $(\diamond)$
 we obtain $\zeta_2^{m-1}=1$.
Replacing $x_i$ by $\zeta_1 x_j$ and $x_i+ \alpha$ by $\zeta_2 (x_j+ \alpha)$
 in the left hand side of 
$(\diamond)$, 
we obtain $(1+ \zeta_1)x_j^m=(1+ \zeta_2)(x_j+\alpha)^m$, so
$(1+\zeta_1)/(1+\zeta_2)=(x_j+\alpha)/x_j$, and $\left( (1+\zeta_1)/(1+\zeta_2) \right)^{m-1} = 1$.
To obtain the claimed expressions of $x_i$ and $x_j$, 
one can replace $x_j$ by 
$\zeta_1^{-1} x_i$ in the equality
$x_i+\alpha=\zeta_2(x_j+\alpha)$.
The converse follows from straightforward computations.
\end{proof}

\begin{example}\label{example_i}
It is straightforward to see that the integers $m=2^k+1$ for $k\geqslant 1$ belong to $\mathcal M$ since
$1$ is the only root of $x^{2^k}+1$.
\end{example}

\begin{remark}\label{critere_M}
As a consequence of Proposition \ref{proposition:condition_xm}
an odd integer $m$ belongs to $\mathcal{M}$
if and only if 
$2(m-1)+1$ does.
It implies that if an 
odd
integer $m$ belongs to $\mathcal{M}$ then 
for all $k\geqslant 0$ the integer $2^k(m-1) +1$ does. 
We also notice that
if an integer $m$ (not necessary odd)
satisfy the condition of Proposition \ref{proposition:condition_xm}
then $2(m-1)+1$ is an element of $\mathcal{M}$.
\end{remark}

\begin{example}\label{4_in_M}
As the polynomial
$x^3-1$ has exactly two roots $\zeta$ and $\zeta^{-1}$ different from the unity,
 we can deduce that $m=4$ satisfies
 the condition of Proposition \ref{proposition:condition_xm}.
Thus according to the above remark,  the integers $2^k3+1$ belong to $\mathcal M$ for $k\geqslant 1$.
\end{example}

\begin{example}\label{list_integers_m}
Proposition \ref{proposition:condition_xm} also provides us
with a method to check 
if an odd integer $m$ belongs  to $\mathcal M$.
For a fixed odd integer $m$, write $m-1=t2^{s}$ with $t$ odd. Hence the $(m-1)$-th roots of unity are exactly the $t$-th roots of unity in characteristic two. Consider the smallest integer $n$ such that $2^n\equiv 1\pmod t$ and compute the list  of the $t$-th roots of unity distinct from $1$ in  the field ${\mathbb F}_{2^n}$. Then  check for $\zeta_1$ and $\zeta_2$ in this list if 
$\left( \frac{1+\zeta_1}{1+\zeta_2}\right)^{t} = 1$
 imply $\zeta_1=\zeta_2$ or $\zeta_1=\zeta_2^{-1}$ using an exhaustive method.
For example
using the open source computer algebra system SAGE
we have determined that
the only odd integers
 less than 200 which do not belong to $\mathcal M$ are
15, 29, 31, 43, 57, 61, 63, 71, 85, 91, 99, 103, 113, 121, 125, 127, 141, 147, 151, 155, 169, 171, 179, 181, 183, 187 and 197.
\end{example}

We give below some infinite families of good exponents.

\begin{example}\label{example_unutilised}
Let us prove that for any $k \geqslant 0$ the integers
 $m=2^k+2$ 
 satisfy the conditions of Proposition \ref{proposition:condition_xm}.
First notice  that 
if $\zeta$  is a $(m-1)$-th root of  unity then
 $(1+\zeta)^{2^{k}+1}=\zeta+\zeta^{-1}$.
As a consequence,
if $\zeta_1$ and $\zeta_2$ are two $(m-1)$-th roots of unity such that 
$\left( \frac{1+\zeta_1}{1+\zeta_2}\right)^{m-1} = 1$
then 
$$\zeta_2 \left( (1+\zeta_1)^{2^{k}+1} + (1+\zeta_2)^{2^{k}+1} \right)
= \zeta_2^{2}+(\zeta_1+\zeta_1^{-1})\zeta_2+1.$$
But this is equal to zero, so 
 $\zeta_2$ is equal to $\zeta_1$ or $\zeta_1^{-1}$.
\end{example}

\begin{example}\label{example_ii}
Applying Remark \ref{critere_M} to the previous example 
we deduce that for any $k$ and $s$ satisfying  $k \geqslant s \geqslant 1$ 
the integer
 $2^k+2^{s}+1$ belongs to $\mathcal{M}$.
\end{example}

\begin{example}
In the case where $m=2^k-1$, with $k \geqslant 4$, 
we notice that for any choice of $\zeta_1$ a ($2^{k-1}-1$)-th root of  unity, 
we also have $(1+\zeta_1)^{2^{k-1}-1}=1$.
So any choice of a couple $(\zeta_1,\zeta_2)$ of ($2^{k-1}-1$)-th 
roots of  unity
such that $\zeta_1\neq\zeta_2$ and $\zeta_1 \zeta_2 \neq 1$
will satisfy the hypothesis 
$\zeta_1^{m-1}  = \zeta_2^{m-1} = 
\left( \frac{1+\zeta_1}{1+\zeta_2}\right)^{m-1} = 1
$
but will not satisfy the conclusion. 
In this case $L_{\alpha} (x^m)$ does not have distinct critical values
so $m \notin \mathcal{M}$.
\end{example}

The following result 
will be our main tool
to obtain infinite families of good exponents
with convenient congruence.
Indeed this result combined with the characterization of the set $\mathcal M$ given in Proposition \ref{proposition:condition_xm}
will provide us the  families of good exponents explicited in Proposition \ref{families_in_M} $(iii)$ and
exploited in Corollaries \ref{First_family} and \ref{Second_family}.

\begin{proposition}\label{Felipe_exponents}
Let $p,\ell$ be distinct primes such that 
$\ell \ne 2, p^{\ell-1} \not\equiv 1 \pmod{\ell^2}$
and that,
if $\z_1,\z_2\ne 1$ are $\ell$-th roots of unity in characteristic $p$ such that $(\z_1+1)/(\z_2+1)$ is also a
$\ell$-th root of unity, then $\z_1=\z_2$ or $\z_1 = \z_2^{-1}$.
Then, for any $k \ge 2$, 
if $\z_1,\z_2\ne 1$ are $\ell^k$-th roots of unity in characteristic $p$ such that $(\z_1+1)/(\z_2+1)$ is also a
$\ell^k$-th root of unity, then $\z_1=\z_2$ or $\z_1 = \z_2^{-1}$.
\end{proposition}

\begin{proof}
Induction on $k$. The case $k=1$ is the hypothesis.

Assume now that $\z_1$ have order exactly $\ell^k, k\ge 2$ and let $\mathbb{F}_q = \mathbb{F}_p(\z_1)$.
Because we assumed that $p^{\ell-1} \not\equiv 1 \pmod{\ell^2}$, we have that the order of $p \pmod{\ell^k}$ is $\ell$
times the order of $p \pmod{\ell^{k-1}}$. Let $\mathbb{F}_r = \mathbb{F}_p(\z_1^{\ell})$. It follows that
$[\mathbb{F}_q:\mathbb{F}_p] = \ell[\mathbb{F}_r:\mathbb{F}_p]$.
Then $q=r^\ell$ and the minimal polynomial of $\z_1$ over $\mathbb{F}_r$ is $x^\ell-\a_1$, where $\a_1 = \z_1^\ell$
has order $\ell^{k-1}$.
In particular $\N \z_1 = \a_1, \Tr \z_1 = 0$ and $\N(1+\z_1) = 1 +\a_1$ where $\N,\Tr$ are respectively the norm and
trace $\mathbb{F}_q/\mathbb{F}_r$, and the last equality follows by evaluating
$x^\ell-\a_1$ at $x=-1$.

Assume first that $\z_2$ have order exactly $\ell^k$ also and that $\z_3=(\z_1+1)/(\z_2+1)$ is also a
$\ell^k$-th root of unity and write $\z_i^{\ell} = \a_i, i=2,3$ so the $\a_i$ are $\ell^{k-1}$-th roots of unity.
As before, we get that $\N \z_i = \a_i, i=2,3$ and that $\N(1+\z_2) = 1 +\a_2$.
Taking norms, we get $\a_3=(\a_1+1)/(\a_2+1)$, so by induction we get that $\a_1=\a_2$ or $\a_1=\a_2^{-1}$.

If $\a_1=\a_2$, then $\a_3=1$ and either $\z_1=\z_2$ as we wanted or $\z_1 = \w\z_2$ with $\w$ of order $\ell$.
In the latter case we get $(1+\w\z_2)/(1+\z_2)=\w^j$ 
for some $j=0,1,\ldots,\ell-1$. If $j \ne 1$, we can solve the equation for
$\z_2$ and get $\z_2 \in \mathbb{F}_p(\omega)$ which is a contradiction. If $j=1$ we get $\w=1$, also a contradiction.

If $\a_1=\a_2^{-1}$, then $\a_3=\a_1$ and either $\z_1=\z_2^{-1}$ as we wanted or $\z_1 = \w\z_2^{-1}$ 
with $\w$ of order $\ell$. In the latter case we get $(1+\z_1)/(1+\w\z_1^{-1})=\w^j\z_1$ for some $j=0,1,\ldots,\ell-1$. 
This gives, for $j \ne 0$, 
$\z_1 \in \mathbb{F}_p(\omega)$ which is a contradiction. For $j=0$, this gives
$\omega = 1$, also a contradiction.

Finally, assume that $\z_2$ have order smaller than $\ell^k$, so $\z_2 \in \mathbb{F}_r$. We write our 
equation as $(\z_1+1)=\z_3(\z_2+1)$. First note that $\z_3$ cannot be in $\mathbb{F}_r$, since $\z_1$ is
not in $\mathbb{F}_r$, so $\Tr \z_1=\Tr \z_3 =0$, so taking trace of our equation gives $1 = 0(\z_2+1)=0$,
contradiction.

\end{proof}

\begin{example}\label{Felipe_computations}
We verified by computer calculation that the hypothesis of this proposition holds when $p=2$ and $\ell < 200$ except for
$\ell = 7,31,73,89,127$. 
For example the case $\ell = 3$ follows from Example \ref{4_in_M}. These computations will enable us to exhibit the examples of 
Corollaries \ref{First_family} and \ref{Second_family}.
\end{example}


\section{Regular extensions}\label{Regular}
Let $n$ be an integer $\geqslant 1$ and set $q=2^n$.
Let $t$ be an element transcendental over  ${\mathbb F}_q$ and $K$ an extension field of ${\mathbb F}_q(t)$. Recall that the extension 
$K/{\mathbb F}_q(t)$ is said to be regular if it is separable and if ${\mathbb F}_q$ is algebraically closed in $K$ i.e. 
${\mathbb F}_q^{K}={\mathbb F}_q$ where  ${\mathbb F}_q^{K}$ is the algebraic closure of  $\mathbb F_q$ in $K$.

Let $\alpha\in{\mathbb F}_q^{\ast}$, let $m$ be an integer and $d=(m-1)/2$ if $m$ is odd and $d=(m-2)/2$ if $m$ is even.
Fix $f\in{\mathbb F}_q[x]$  a polynomial of degree $m$ such that
 the associated polynomial 
$L_{\alpha}f$ 
has degree exactly $d$.
Furthermore, we suppose that $d$ is odd which is equivalent to say that $m\equiv 0\pmod 4$ or $m\equiv 3\pmod 4$.

\subsection{First floor: monodromy}

We consider the arithmetic monodromy group $G$ of the polynomial $L_{\alpha}f$.
It is the Galois group of the extension $F/{\mathbb F}_q(t)$ where $F$ is the splitting field of the
polynomial $L_{\alpha}f(x)-t$ over the field ${\mathbb F}_q(t)$.
Consider also $\overline G:=\Gal(F/{\mathbb F}_q^{F}(t))$
the geometric monodromy group of $L_{\alpha}f$.
The groups $G$ and $\overline G$ are transitive subgroups of the symmetric group $\Sgoth_d$ and  $\overline{G}\lhd G$.

\begin{proposition}\label{Monodromy}
Let $f\in{\mathbb F}_q[x]$ be a polynomial such that
 the associated polynomial 
$L_{\alpha}f$ is Morse and
has (odd) degree $d$.

\begin{enumerate}[label=(\roman*)]
\item Let $u$ be a root of $L_{\alpha}f(x)-t$ in $F$. Then,
for each place $\wp$ of $F$ above the place $\infty$ at infinity
of ${\mathbb F}_q(t)$, we have
 that $u$ has a simple pole at $\wp$.
\item The group $\Gal(F/{\mathbb F}_q(t))$ is the full symmetric group $\Sgoth_d$ and the extension $F/{\mathbb F}_q(t)$ is regular.
\end{enumerate}
\end{proposition}

\begin{proof}
If $v_{\wp}$ is the valuation at the place $\wp$,
we have
$v_{\wp}(L_{\alpha}f(u))=v_{\wp}(t)$
and by definition of the ramification index $e \left( \wp \vert \infty \right)$ we have
$v_{\wp}(t)=e \left( \wp \vert \infty \right)  v_{\infty} (t)=-e \left( \wp \vert \infty \right)$.
 Since $d$ is supposed to be odd, it is prime to the characteristic of ${\mathbb F}_q(t)$, and
 then, by the proof of Theorem 4.4.5  of \cite{Serre}, we have 
$e \left( \wp \vert \infty \right)=d$. Hence, we obtain $v_{\wp}(L_{\alpha}f(u))=-d$, which implies that
 $v_{\wp}(u)=-1$ and thus $u$ has a simple pole at $\wp$. 

The analogue of the Hilbert theorem given by Serre in Theorem 4.4.5 of \cite{Serre} and detailled in even characteristic in the appendix of Geyer in \cite{JardenRazon} gives that  the 
geometric monodromy group $\Gal(F/{\mathbb F}_{q}^F(t))$ of $L_{\alpha}f$
is the symmetric group
$\Sgoth_d$. But it is contained in the arithmetic monodromy group $\Gal(F/{\mathbb F}_q(t))$ which is also a subgroup of $\Sgoth_d$. So they are equal and ${\mathbb F}_{q}^F={\mathbb F}_q$.
\end{proof}

A consequence of the first part of the previous proposition is that 
$L_{\alpha}f(x)-t$ has only simple roots; let us call them $u_0,\ldots,u_{d-1}$.

 \subsection{Second floor}\label{section:SecondFloor}

Let $x_i$ such that 
$x_i^2+\alpha x_i=u_i$.
Hence we have $D_{\alpha}f(x_i)=t$.
Consider $\Omega={\mathbb F}_q(x_0,\ldots, x_{d-1})$  the compositum of the fields $F(x_i)$ and  ${\mathbb F}_q^{\Omega}F$  the compositum of 
$F$ and ${\mathbb F}_q^{\Omega}$.
Let  $\Gamma=\Gal(\Omega/ F)$ and $\overline\Gamma=\Gal(\Omega/{\mathbb F}_q^{\Omega}F)$.

The following statement appears in \cite{Felipe}.

\begin{lemma}\label{lemma_sum_u_i}
Suppose
that $L_{\alpha}f$ is Morse and has degree $d$.
If $J \subset \{ 0, \ldots, d-1 \}$ is neither empty nor the whole set then
$ \sum_{j \in J} u_j $  has a pole at a place of $F$ over the  place
$\infty$ of ${\mathbb F}_q(t)$.
\end{lemma}
\begin{proof}
To obtain a contradiction suppose that 
$J \subset \{ 0, \ldots, d-1 \}$ is such that
$j_0 \in J$ whereas $j_1 \in \{ 0, \ldots, d-1\} \setminus J$.
Suppose also that 
$ \sum_{j \in J} u_j $ has no pole in places above $\infty$.
Then it has no pole at all, and so it is constant.
Recall that 
$\Gal \left( F/ {\mathbb F}_q(t) \right)$
is  $\Sgoth_d$ 
by Proposition \ref{Monodromy}.
Applying to $ \sum_{j \in J} u_j $ 
the automorphism corresponding to the transposition 
$(j_0j_1) \in  \Sgoth_d$ 
 one obtains 
 $\sum_{j \in J \setminus\{ j_0\}} u_j + u_{j_0}=\sum_{j \in J \setminus\{ j_0\}} u_j + u_{j_1}$,
 which leads to $u_{j_0}=u_{j_1}$, a contradiction.
\end{proof}

\begin{lemma}\label{lemma_sum_x_i}
Suppose
that $L_{\alpha}f$ is Morse and has degree $d$.
Let $\widetilde F$ be $F$ or ${\mathbb F}_q^{\Omega}F$.
Let $J$ be a non-empty subset of $\{0, \ldots, d-1 \}$
different from $\{0, \ldots, d-1 \}$. Then
$$\sum_{j \in J} x_j \notin \widetilde F.$$
\end{lemma}
\begin{proof}
To obtain a contradiction,
suppose that $\sum_{j \in J} x_j \in \widetilde F$.
By Lemma \ref{lemma_sum_u_i} we know that there exists a place $\wp$ of $F$ above $\infty$ such that
 $\sum_{j \in J} u_j$  has a pole at $\wp$.
Moreover, this pole is simple as for all $j\in \{0, \ldots, d-1 \} $ 
the root $u_j$ has a simple pole by Proposition \ref{Monodromy}.
Now consider 
$A = \left( \sum_{j \in J} x_j  \right)$ and $B=  \left( \sum_{j \in J} x_j  + \alpha \right)$.
If $A$ (and thus $B$) belongs to $\widetilde F$,
one can consider the valuation of $A$ and $B$ at $\wp$.
As $A.B = \sum_{j \in J} u_j $ it follows that
either $A$ or $B$ has a pole.
Since $A$ and $B$ differ from a constant,
$A$ has a pole if and only if $B$ has a pole.
So both have a pole and  
the order of multiplicity 
is the same.
Then we obtain  
$2v_{\wp}(A)=-1$, a contradiction.
\end{proof}

\begin{lemma}\label{lemma_compositum_of_two_extensions}
Let $k(x_1)$ and $k(x_2)$ be two Artin-Schreier extensions
of a field $k$ of characteristic 2.
Suppose that 
 $x_i^2+\alpha x_i =w_i$ with $\alpha$ and $w_i$ in  $k^{\ast}$.
Then   $k(x_1)=k(x_2)$  if and only if $x_1 + x_2 \in k$.

Moreover if $x_1 + x_2 \notin k$ then $k(x_1,x_2)$ is a degree 4 extension of $k$ 
and the three fields lying between 
$k$ and $k(x_1,x_2)$ are those of the following diagram.
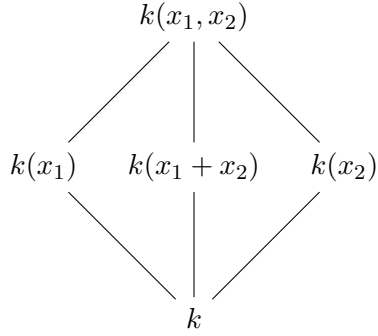
\begin{center}
\begin{tikzpicture}[node distance=2cm]
 \node (k)                  {$k$};
 \node (k1plus2) [above of=k]  {$k(x_1+x_2)$};
 \node (k12)  [above of=k1plus2]   {$k(x_1,x_2)$};
 \node (k2)  [right of=k1plus2]  {$k(x_2)$};
 \node (k1)  [left of=k1plus2]   {$k(x_1)$};
 \draw (k)   -- (k2);
 \draw (k)   -- (k1plus2);
 \draw (k)   -- (k1);
 \draw (k2)  -- (k12);
 \draw (k1)  -- (k12);
 \draw (k1plus2)  -- (k12);
\end{tikzpicture}
\end{center}
\end{lemma}

\begin{proof}
For the first assertion, see the proof of Lemma 4.1 in \cite{YvesFabien}.
In the case where $x_1+x_2 \notin k$, 
we can use $[k(x_1)(x_2):k(x_1)]=2$ to prove 
$[k(x_1,x_2):k]=4$.
We
 deduce 
 that $\Gal(k(x_1,x_2)/k) = \left( {\mathbb Z} / 2{\mathbb Z} \right)^2$.
The field 
$k(x_1+x_2)$ is a subextension since
$x_1+x_2$ is a root of $x^2+\alpha x = w_1 +w_2$.
It remains to prove that $k(x_1+x_2)$ is different from
$k(x_1)$ (and $k(x_2)$). 
According to the first statement of the lemma,
it is sufficient to check 
that $x_1+(x_1 + x_2) \notin k$.
\end{proof}

\begin{proposition}\label{Tour_infernale}
Suppose
that $L_{\alpha}f$ is Morse and has degree $d$.
Let $\widetilde F$ be $F$ or ${\mathbb F}_q^{\Omega}F$.
Let $r$ be an integer such that $0\leqslant r\leqslant d-2$. Then 
\begin{enumerate}[label=(\roman*)]
\item the field $\widetilde F(x_0, \ldots,x_r)$ is an extension of order $2^{r+1}$ of $\widetilde F$,
\item the Galois group $\Gal \left( \widetilde F(x_0, \ldots,x_r) / \widetilde F \right)$ 
 is $\left( {\mathbb Z} / 2 {\mathbb Z}\right)^{r+1}$ and
\item there are $2^{r+1}-1$ quadratic extensions of $\widetilde F$ between
$\widetilde F$ and $\widetilde F(x_0, \ldots,x_r)$. 
Namely, these extensions are 
the extensions $\widetilde F \left( \sum_{j \in J} x_j \right) $
with non-empty $J \subset \{0, \ldots, r \}$.
\end{enumerate}
\end{proposition}

\begin{proof}
We proceed by induction.
The case $r=0$ is trivial and the case $r=1$ is given 
by Lemma \ref{lemma_compositum_of_two_extensions}.
Assuming that the proposition holds for $r-1$, with $1\leqslant r\leqslant d-2$,
we will prove it for $r$.
The main idea is to consider the extensions of the following diagram
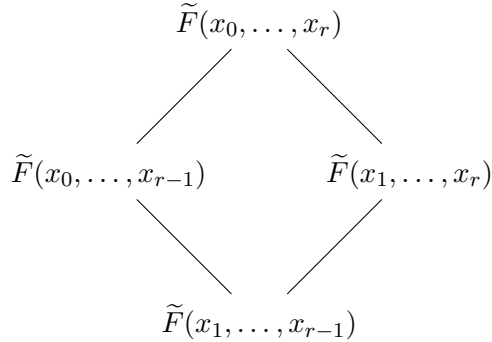
\begin{center}
\begin{tikzpicture}[node distance=2cm]
 \node (k)                  {$\widetilde F(x_1,\ldots,x_{r-1})$};
 \node (k1plus2) [above of=k]  {};
 \node (k12)  [above of=k1plus2]   {$\widetilde F(x_0,\ldots,x_{r})$};
 \node (k2)  [right of=k1plus2]  {$\widetilde F(x_1,\ldots,x_{r})$};
 \node (k1)  [left of=k1plus2]   {$\widetilde F(x_0,\ldots,x_{r-1})$};
 \draw (k)   -- (k2);
 \draw (k)   -- (k1);
 \draw (k2)  -- (k12);
 \draw (k1)  -- (k12);
\end{tikzpicture}
\end{center}
and to apply Lemma \ref{lemma_compositum_of_two_extensions}.
We first prove that $x_0+x_r \notin \widetilde F(x_1,\ldots,x_{r-1})$.
Otherwise we would have the quadratic extension
$\widetilde F(x_0+x_r)$
between $\widetilde F$ and $\widetilde F(x_1, \ldots, x_{r-1})$.
By the induction hypothesis,
there would exist $J \subset \{ 1 , \ldots , r-1\} $ such that
$\widetilde F(x_0+x_r)=\widetilde F \left( \sum_{j \in J} x_j \right)$.
By Lemma \ref{lemma_compositum_of_two_extensions} again
we would have $x_0+x_r+ \sum_{j \in J} x_j \in \widetilde F $ and then a contradiction
with Lemma \ref{lemma_sum_x_i}.
Then we can apply the conclusions of Lemma \ref{lemma_compositum_of_two_extensions}
with $k=\widetilde F(x_1,\ldots,x_{r-1})$
to obtain that $\widetilde F(x_0,\ldots,x_{r})$ is a quadratic extension of
both $\widetilde F(x_1,\ldots,x_{r})$ and $\widetilde F(x_0,\ldots,x_{r-1})$.
It follows that 
$[\widetilde F(x_0,\ldots,x_{r}):\widetilde F]=2^{r+1}$.

Furthermore, we can define $2^{r+1}$ different $\widetilde F$-automorphisms of $\widetilde F(x_0,\ldots,x_{r})$ by sending
$x_i$ to $x_i$ or to $x_i+\alpha$. 
So, all the elements of the Galois group $\Gal \left( \widetilde F(x_0, \ldots,x_r) / \widetilde F\right)$ have order dividing 2
thus this group is certainly $\left( {\mathbb Z} / 2{\mathbb Z} \right)^{r+1}$.

 For any non-empty subset $J \subset \{ 0, \ldots, r\}$
 we see that $\sum_{j \in J} x_j$ is a root
 of $x^2+\alpha x = \sum_{j \in J} u_j$, and we know
 from Lemma  \ref{lemma_sum_x_i} that $\sum_{j \in J} x_j \notin \widetilde F$.
We  obtain this way $2^{r+1}-1$ different quadratic extensions between $\widetilde F$ and 
 $\widetilde F(x_0, \ldots,x_r)$. 
 Indeed, we can show that these extensions are different. 
 If $\widetilde F \left( \sum_{j \in J_1} x_j \right) = \widetilde F \left( \sum_{j \in J_2} x_j \right)$
 then $ \sum_{j \in J_1} x_j  + \sum_{j \in J_2} x_j  \in \widetilde F$ which leads to $J_1=J_2$ 
 using Lemma \ref{lemma_sum_x_i}.
 Finally, these $2^{r+1}-1$ quadratic extensions are the only ones.
Indeed, the quadratic extensions between $\widetilde F$ and 
 $\widetilde F(x_0, \ldots,x_r)$ are in correspondence
 with the subgroups of $\left( {\mathbb Z} / 2{\mathbb Z} \right)^{r+1}$
 of index $2$.
These subgroups are 
the hyperplanes of  $\left( {\mathbb Z} / 2 {\mathbb Z}\right)^{r+1}$
and one can count $2^{r+1}-1$ of them.
\end{proof}

\begin{proposition}\label{Galois_second_floor}
Suppose
that $L_{\alpha}f=\sum_{k=0}^db_{d-k}x^{k}$ is Morse and has degree $d$.
Let $\widetilde F$ be $F$ or $F{\mathbb F}_q^{\Omega}$.
If there exists $x \in {\mathbb F}_q$ such that 
$x^2+ \alpha x = b_1/b_0$ then 
$\Gal \left( \widetilde F(x_0, \ldots,x_{d-1}) / \widetilde F \right)$ 
 is $\left( {\mathbb Z} / 2{\mathbb Z} \right)^{d-1}$ and thus the extensions $\Omega/F$ and  $\Omega/{\mathbb F}_q(t)$ are regular.
\end{proposition}

\begin{proof}
As Proposition \ref{Tour_infernale} already gives
$\Gal \left( \widetilde F(x_0, \ldots,x_{d-2}) / \widetilde F \right)
= \left( {\mathbb Z} / 2 {\mathbb Z} \right)^{d-1}$,
it remains to study the extension
$ \widetilde F(x_0, \ldots,x_{d-1}) / 
\widetilde F(x_0, \ldots,x_{d-2}) $.

Using $\sum_{i=0}^{d-1} u_i = b_1 / b_0$ and 
the linearity of $x\mapsto x^2+ \alpha x$, we see that
in any case the equation $x^2+\alpha x = b_1 / b_0$
has two solutions in $\overline{\mathbb F}_q$, namely
 $\sum_{i=0}^{d-1} x_i$  and $\alpha + \sum_{i=0}^{d-1} x_i$.
With our hypothesis we deduce that 
 $\sum_{i=0}^{d-1} x_i \in {\mathbb F}_q$ hence
$ \widetilde F(x_0, \ldots,x_{d-1}) =
\widetilde F(x_0, \ldots,x_{d-2}) $ and the result about the Galois group follows.
Thus we have proved that $\Gamma=\overline\Gamma$ and then $\Omega/F$ is regular.
 Proposition \ref{Monodromy} shows that the extension $F/{\mathbb F}_q(t)$ is regular, hence 
 we deduce the regularity of the extension $\Omega/{\mathbb F}_q(t)$.
\end{proof}

\section{Main results}\label{Main}

The main ingredient of the proof of our main results is the 
Chebotarev density theorem. The next proposition summarizes its contribution
in our context.

\begin{proposition}\label{application_Chebotarev}
Let $m \geqslant 7$ be an integer such that $m\equiv 3\pmod 4$.
Then 
there exists an integer $N$ depending only on  $m$
such that for all $n \geqslant N$, if we set $q=2^n$,
for all $f \in \mathbb{F}_q[x]$ of degree   $m$,
and for all  $\alpha$  in $\mathbb{F}_q^*$
such that
the extension $\Omega/{\mathbb F}_q(t)$ is regular,
there exists $\beta\in {\mathbb F}_q$ such that the polynomial 
$D_{\alpha}f(x)+\beta$ 
splits in $ \mathbb{F}_q[x]$ with no repeated factors. 

\end{proposition}

\begin{proof}
As $m\equiv 3\pmod 4$, 
by Proposition \ref{degree_of_g}
the polynomial $L_{\alpha}f$ has  degree exactly $d=(m-1)/2$, which is odd by our hypothesis on $m$, and thus $F/{\mathbb F}_q(t)$ is separable. Since
the extension $\Omega/F$  is also separable we obtain that $\Omega/\mathbb{F}_q(t)$ is  separable and thus Galois.

Since the extension $\Omega/\mathbb{F}_q(t)$ is supposed to be regular,
by an application of the  Chebotarev theorem 
(see Theorem 1 in \cite{FouqueTibouchi}
which is deduced  from Proposition 4.6.8 in \cite{Rosen})
the number $N(S)$ of 
places $v$ of ${\mathbb F}_q(t)$
of degree 1 unramified in $\Omega$ and
such that the Artin symbol  $\left(\frac{\Omega/{\mathbb F}_q(t)}{v}\right)$ is equal to the conjugacy class of 
$\Gal(\Omega/{\mathbb F}_q(t))$ consisting of the identity element
satisfies
$$N(S)\geqslant \frac{q}{d_{\Omega}}-2 \Bigl((1+\frac{g_{\Omega}}{d_{\Omega}})q^{1/2}+q^{1/4}+1+\frac{g_{\Omega}}{d_{\Omega}}\Bigr)$$
where $d_{\Omega}:=[\Omega:\mathbb{F}_q(t)]$ and $g_{\Omega}$ is the genus of $\Omega$.

But we have seen that $G=\Gal(F/{\mathbb F}_q(t))$ is a subgroup of $\Sgoth_d$  
and $\Gamma=\Gal(\Omega/F)$ is a group of order bounded by $2^d$, thus we 
have
 $d_{\Omega} \leqslant d! 2^{d}$.
Moreover, one can obtain an upper bound on $g_{\Omega}$ depending only on $d$
using Lemma 14 of \cite{Pollack}  to get that:
$g_{\Omega}\leqslant \frac{1}{2}(\deg D_{\alpha}f-3)d_{\Omega}  +1$
i.e.
$$g_{\Omega}\leqslant (d!2^d)\times (d-3/2)+1.$$
Then if $q$ is sufficiently large we will have $N(S)\geqslant 1$ 
which concludes the proof.
\end{proof}

Since the methods of our proofs need  the degree $m$ of the polynomials to belong to the set $\mathcal M$
defined in Definition \ref{etlabete}, 
we sum up some infinite subsets of $\mathcal M$ we have pointed out in Subsection \ref{EnsembleM}.

\begin{proposition}\label{families_in_M}
The following integers $m$ belong to the set $\mathcal M$:
\begin{enumerate}[label=(\roman*)]
\item $m=2^k+1$ for $k\geqslant 1$.
\item  $m=2^k+2^{s}+1$ for $k \geqslant s \geqslant 1$.
\item $m=2^s\ell^k+1$ 
for  $k \geqslant 1$, $s\geqslant 1$ and for $\ell$ an odd prime 
such that $2^{\ell-1} \not\equiv 1 \pmod{\ell^2}$ and such that $m':=\ell+1$ 
satisfy the condition of Proposition \ref{proposition:condition_xm}.
\end{enumerate}
\end{proposition}

\begin{proof}
The  first two assertions are proved respectively in Example \ref{example_i} and \ref{example_ii}.
If $\ell$ satisfy the hypothesis $(iii)$ then Proposition \ref{Felipe_exponents} in the case of characteristic two  tells us that ${\ell}^k+1$ also satisfy the condition of Proposition \ref{proposition:condition_xm}. Now use Remark \ref{critere_M} to have that $2^s\ell^k+1$ satisfy the condition of Proposition \ref{proposition:condition_xm}. For $s\geqslant 1$ it is odd and so it belongs to $\mathcal M$.
\end{proof}

Now we can state and prove our main results
which establish for some polynomials
$f$ the maximality of the differential uniformity $\delta(f)$
defined in Section 1 by 
$\displaystyle{\delta(f)=\max_{(\alpha,\beta)\in{\mathbb F}_q^{\ast}\times{\mathbb F}_q}\sharp\{x\in{\mathbb F}_{q} \mid f(x+\alpha)+f(x)=\beta\}.}$

\begin{theorem}\label{Principal_7_mod_8}
Let  $m\in{\mathcal M}$ such that $m\equiv 7\pmod 8$.
Then for $n$ sufficiently large,
for all polynomials $f\in{\mathbb F}_{2^n}[x]$
of degree $m$ we have  $\delta(f)=m-1$.
\end{theorem}

\begin{proof}
We fix $m\in{\mathcal M}$ such that $m\equiv 7\pmod 8$.
Let us prove that for $n$ sufficiently large and
 for any polynomial $\displaystyle{f=\sum_{i=0}^m a_{m-i} x^{i}}$ in ${\mathbb F}_{2^n}[x]$ of degree $m$, 
  there exists  $\alpha$ in ${\mathbb F}_{2^n}^{\ast}$ such that:
 \begin{itemize}[label=--]
 \item  $L_{\alpha}f$ is Morse
 
\item  the equation $x^2+ \alpha x = \frac{b_1}{b_0}$ has a solution in $\mathbb{F}_{2^n}$, where $L_{\alpha}f=\sum_{i=0}^{d} b_{d-i} x^i$.
\end{itemize}

 By Theorem \ref{L_alpha_Morse}, for all $f \in {\mathbb F}_{2^n}[x]$ of degree $m$, the number of elements $\alpha$ in 
${\mathbb F}_{2^n}^{\ast}$ such that $L_{\alpha}f$ is Morse is 
at least 
$2^n-\frac{1}{64}(m-3)(5m^2+28m+7)$.
 
 Moreover, by the Hilbert'90 Theorem, the equation $x^2+ \alpha x = \frac{b_1}{b_0}$
has a solution in $\mathbb{F}_{2^n}$ if and only if
$\textrm{Tr}_{{\mathbb F}_{2^n}/{\mathbb F}_2}  \left( \frac{b_1}{b_0\alpha^2} \right)=0$. By Lemma \ref{lemma:b1overb0} it is equivalent to
 $\textrm{Tr}_{{\mathbb F}_{2^n}/{\mathbb F}_2}   \left( \frac{a_1^2+a_0a_2}{a_0^2\alpha^2} \right)=0$.
In the case where $a_1^2+a_0a_2=0$ every choice of $\alpha\in{\mathbb F}_{2^n}^{\ast}$ is convenient.
 Otherwise the map sending $\alpha$ to $\frac{a_1^2+a_0a_2}{a_0^2\alpha^2}$ is a permutation of  
 ${\mathbb F}_{2^n}^{\ast}$  and then $2^{n-1}-1$ values of $\alpha$ are convenient.
 
 Hence as soon as $2^{n-1}>\frac{1}{64}(m-3)(5m^2+28m+7)+1$ we will have for any 
 $f \in {\mathbb F}_{2^n}[x]$ of degree $m$
  the existence of $\alpha$ in ${\mathbb F}_{2^n}^{\ast}$ satisfying the two conditions.
 Now, these conditions imply by Proposition \ref{Galois_second_floor} that the extension $\Omega/{\mathbb F}_{2^n}(t)$ is regular.
 
 Finally we can apply Proposition \ref{application_Chebotarev} to obtain, for $n$ sufficiently large depending only on $m$, the existence of $\beta\in {\mathbb F}_{2^n}$ such that the polynomial 
$D_{\alpha}f(x)+\beta$ 
splits in $ \mathbb{F}_{2^n}[x]$ with no repeated factors. 
Then $\delta(f)=m-1$.
 \end{proof}

To be concrete, 
using Proposition \ref{families_in_M}, the computations of Example \ref{Felipe_computations}
and taking into account the congruences of $m$
we present in the following corollary
some families
of infinitely many integers for which Theorems \ref{Principal_7_mod_8}  holds.

\begin{corollary}\label{First_family}
Let $\ell$ be a prime congruent to $3$ modulo $4$ such that $2^{\ell-1} \not\equiv 1 \pmod{\ell^2}$ and  $\ell+1$ 
satisfy the condition of Proposition \ref{proposition:condition_xm}
(for example, $\ell \in \{ 
3,11,19,23,43,47, 59,67,71,79,83,103, 107,131,139, 151,163, 167,179,191,\break 199 
\ldots \})$.
Set $m=2{\ell}^{2k+1}+1$ with $k\geqslant 0$.
Then for $n$ sufficiently large,
for all polynomials $f\in{\mathbb F}_{2^n}[x]$
of degree $m$ we have  $\delta(f)=m-1$.
\end{corollary}

When $m$ is congruent to $3$ modulo $8$,
 we also obtain some results but
we have conditions on the parity of $n$ or 
we have to remove some polynomials.

\begin{theorem}\label{Principal_3_mod_8}
Let  $m\in{\mathcal M}$ such that $m\geqslant 7$ and $m\equiv 3\pmod 8$.

\begin{enumerate}[label=(\roman*)]
\item For $n$ even and sufficiently large and for all polynomials $f\in{\mathbb F}_{2^n}[x]$
of degree $m$ we have $\delta(f)=m-1$.
\item For $n$ sufficiently large and for all polynomials $f=\sum_{i=0}^{m} a_{m-i} x^i$ 
in ${\mathbb F}_{2^n}[x]$ of degree $m$ such that $a_1^2+a_0a_2 \neq 0$, we have $\delta(f)=m-1$.
\end{enumerate}
\end{theorem}

\begin{proof}
The proof is similar as the one of  Theorem \ref{Principal_7_mod_8}.
The main difference comes from the expression of $b_1/b_0$ when $m\equiv 3\pmod 8$.
According to Lemma \ref{lemma:b1overb0}, 
we have $\Tr_{{\mathbb F}_{2^n}/{\mathbb F}_2}\left(\frac{b_1}{b_0\alpha^2}\right)=0$ if and only if 
$\Tr_{{\mathbb F}_{2^n}/{\mathbb F}_2}\left(\frac{a_1^2+a_0a_2}{a_0^2\alpha^2}\right)=n$. 
The  arguments of the  above proof apply except when 
$a_1^2+a_0a_2=0$ and $n$ is odd.
\end{proof}

We remark that one could not expect better in the case where 
$m\equiv 3 \pmod 8$,
$a_1^2+a_0a_2 = 0$ and 
$n$ odd since  Theorem 2 (iii) of  \cite{Felipe} gives that 
 $\delta(f)<m-1$ in this case. 
 
Again using Proposition \ref{families_in_M} and the computations of Example \ref{Felipe_computations} we obtain the following corollary.

\begin{corollary}\label{Second_family}
Let $\ell$ be an odd prime such that  $2^{\ell-1} \not\equiv 1 \pmod{\ell^2}$ and
$\ell+1$ 
satisfy the condition of Proposition \ref{proposition:condition_xm}.
\begin{enumerate}[label=(\roman*)]
\item If $\ell\equiv 1 \pmod 8$ then Theorem \ref{Principal_3_mod_8} holds for the integers $m=2\ell^{k}+1$ with $k\geqslant 1$ (for example if $\ell \in \{ 17,41,97,113,137,193,\ldots\}$).
\item If $\ell\equiv 7 \pmod 8$ then Theorem \ref{Principal_3_mod_8} holds for the integers $m=2\ell^{2k+1}+1$ with $k\geqslant 0$ (for example if $\ell \in \{
 23,47,71,79,103,151,167,\break 191,199,\ldots \}$).
 \end{enumerate}
\end{corollary}
 

Finally, we prove  Conjecture \ref{conjecture} when $m\equiv 7\pmod 8$.

\begin{theorem}\label{proof_conjecture}
For a given integer $m\in\mathcal M$ such that $m\equiv 7\pmod 8$, 
there exists
$\varepsilon_m >0$ such that for all sufficiently large $n$,
 if $f$ is a polynomial of degree $m$ over 
${\mathbb F}_{2^n}$, for at least 
$\varepsilon_m 2^{2n}$ values of $(\alpha, \beta) \in   {\mathbb F}^{\ast}_{2^n}\times{\mathbb F}_{2^n}$
we have
$ \sharp \{x\in{\mathbb F}_{q} \mid f(x+ \alpha) + f(x)=\beta\} = \delta(f)=m-1$. 
\end{theorem}

\begin{proof}
We follow the strategy described in the proofs above. The point is to give lower bounds for the number of choices of $\alpha$ and $\beta$.
We have shown the existence of a polynomial $P$ of degree 3 such that 
for any $n$ and any $f\in{\mathbb F}_{2^n}[x]$
there exist at least $2^{n-1} +P(m)$ elements $\alpha$ such that the extension $\Omega/{\mathbb F}_{2^n}(t)$ is regular (see the proof of Theorem \ref{Principal_7_mod_8}).
Thus
for any $\gamma_m<1/2$, for $n$ sufficiently large, there exists $\gamma_m2^n$
 suitable choices of $\alpha$.
For such a choice of $\alpha$, the Chebotarev theorem used in the proof of Proposition \ref{application_Chebotarev} guarantees the existence of 
$\frac{1}{d!2^d}2^n+Q(2^{n/4})$ elements $\beta$ such that $D_{\alpha}f(x)+\beta$ has $\delta(f)$ solutions   where $Q$ is a polynomial of degree 2.
Thus for any $\gamma_m'<1/d!2^d$, for $n$ sufficiently large, there exist $2^n\gamma_m'$ suitable choices of $\beta$.
Hence we obtain the result for any $\varepsilon_m<1/d!2^{d+1}$.
\end{proof}

Remark that the proof of Theorem \ref{proof_conjecture} provides  explicit values of $\varepsilon_m$, namely
any $\varepsilon_m$ between $0$ and $1/d!2^{d+1}$ with $d=\frac{m-1}{2}$.
Remark also that, in the case where $m\equiv 3\pmod 8$, the same strategy leads to a proof of an analogue of  this theorem for polynomials $f$ such that $a_1^2+a_0a_2 \neq 0$ or a proof of another analogue  for even $n$.

\bigskip

\emph{ Acknowledgements:} The third author would like to thank the I2M and CIRM
for support in connection with a number of visits to Luminy and the Simons Foundation for
financial support under grant \#234591.

Moreover, the authors thank the referee for valuable comments.

\bigskip
\bibliographystyle{plain}

\end{document}